\documentclass[11pt]{article}
\usepackage{relsize}
\usepackage{epic}
\usepackage{mathtools}
\usepackage{amsmath, amssymb}
\usepackage{amsthm}
\usepackage{latexsym}
\usepackage{tikz-cd}
\usepackage{fullpage}
\usepackage{xcolor}
\usepackage{MnSymbol} 
\usepackage{amssymb} 
\usepackage{mathtools} 
\usepackage{tikz} 
\usetikzlibrary{chains} 
\usepackage{verbatim} 
\usepackage{longtable} 
\usepackage{titlesec} 
\usepackage{enumerate} 
\usepackage{makecell}
\usepackage{pdflscape} 
\usepackage{titlesec} 
\usepackage{ stmaryrd }

\newtheorem{theorem}{Theorem}[section]
\newtheorem{lemma}[theorem]{Lemma}
\newtheorem{proposition}[theorem]{Proposition}
\newtheorem{corollary}[theorem]{Corollary}

\newcommand{\closure}[2][3]{{}
\mkern#1mu\overline{\mkern-#1mu#2}} %

\titleformat
{\section} 
[block] 
{} 
{\thesection.} 
{1ex} 
{\centering} 
[] 

\titleformat
{\subsection} 
[runin] 
{\bf} 
{\thesubsection.} 
{1ex} 
{} 
[] 

\title{On the Morita Frobenius numbers \\ of blocks of finite reductive groups}
\author{Niamh Farrell}

\begin{document}

\maketitle


\begin{quote} 
ABSTRACT. We show that the Morita Frobenius number of the blocks of the alternating groups, the finite groups of Lie type in describing characteristic, and the Ree and Suzuki groups is 1. We also show that the Morita Frobenius number of almost all of the unipotent blocks of the finite groups of Lie type in non-defining characteristic is 1, and that in the remaining cases it is at most 2. 
\end{quote}


\section{INTRODUCTION}

Let $\ell$ be a prime number, let $k = \overline{\mathbb{F}}_\ell$ be an algebraic closure of the field of $\ell$ elements and let $A$ be a finite dimensional $k$-algebra. For $a \in \mathbb{N}$, the \textit{$a$-th Frobenius twist of $A$}, denoted by $A^{(\ell^a)}$, is a $k$-algebra with the same underlying ring structure as $A$, endowed with a new action of the scalars of $k$ given by $\lambda.x = \lambda^{\frac{1}{\ell^a}}x$ for all $\lambda \in k$, $x \in A$. Two finite dimensional algebras $A$ and $B$ are \textit{Morita equivalent} if mod$(A)$ and mod$(B)$ are equivalent $k$-linear categories. By definition, $A$ and $A^{(\ell^a)}$ are isomorphic as rings, however, they need not even be Morita equivalent as $k$-algebras. The \textit{Morita Frobenius number} of a $k$-algebra $A$, denoted by $mf(A)$, is the least integer $a$ such that $A$ is Morita equivalent to $A^{(\ell^a)}$. 

The concept of Morita Frobenius numbers was introduced by Kessar in \cite{K2} in the context of Donovan's Conjecture in block theory. Donovan's Conjecture implies that Morita Frobenius numbers of $\ell$-blocks of finite groups are bounded by a function which depends only on the size of the defect groups of the block. Little is known about the values of Morita Frobenius numbers in general, but it is known that a block of a group algebra can have Morita Frobenius number greater than 1 \cite{B/K}. In this paper we calculate the Morita Frobenius numbers of a large class of blocks of finite reductive groups. We have used GAP \cite{GAP4} to check that the Morita Frobenius number of blocks of simple sporadic groups and their covers is 1. See Sections~\ref{sec:grouptheory} to~\ref{sec:unip} for an explanation of the notation in the following theorem.

\begin{theorem}
\label{thrm:maintheorem}
Let $b$ be an $\ell$-block of a quasi-simple finite group $G$. Let $\closure{G} = G/ Z(G)$. Suppose that one of the following holds.
\begin{enumerate}[(a)]
\item $\closure{G}$ is an alternating group
\item $\closure{G}$ is a finite group of Lie type in defining characteristic
\item $\closure{G}$ is a finite group of Lie type in non-defining characteristic, $b$ dominates a unipotent block of $\closure{G}$, and $b$ is not one of the following blocks of $E_8$
	\begin{itemize}
		\item $b = b_{E_8}(\phi_1^2.E_6(q), {E_6}[\theta^i]) $ $ (i=1,2) $ with $\ell=2$ and  $q \equiv 1$ modulo 4 
		\item $b = b_{E_8}({\phi_2^2}.{^2E}_6(q), {^2E_6}[\theta^i])$ $(i=1,2) $ with $\ell \equiv 2$ mod $3$ and $q \equiv 2$ modulo $\ell$
	\end{itemize}
\end{enumerate}
Then $mf(b) = 1$. In the excluded cases of part (c), $mf(b)\leq 2$. 
\end{theorem}

We start with some general results on the Morita Frobenius numbers of blocks in Section~\ref{sec:grouptheory}. Section~\ref{sec:AltGps} deals with the case of the alternating groups, and Section~\ref{sec:definingchar} deals with finite groups of Lie type in defining characteristic. In Section~\ref{sec:unip} we first present key results from $e$-Harish Chandra theory and unipotent block theory, followed by the results for finite groups of Lie type in non-defining characteristic. Finally, Section~\ref{sec:mainproof} contains the proof of Theorem~\ref{thrm:maintheorem}.

\section{GENERAL RESULTS ON MORITA FROBENIUS NUMBERS OF BLOCKS}
\label{sec:grouptheory}

Throughout, $\ell$ is a prime number, $k$ is an algebraically closed field of characteristic $\ell$, and $G$ is a finite group.

\subsection{Results on $k$-algebras.}
\label{subsec:kalgs}

Let $A$ and $B$ be finite dimensional $k$-algebras and let $A_0$ and $B_0$ be basic algebras of $A$ and $B$ respectively. We define the \textit{Frobenius number} of $A$ to be the the least integer $a$ such that $A \cong A^{(\ell^a)}$ as $k$-algebras, and denote it by $frob(A)$. Recall that $A$ and $B$ are Morita equivalent if and only if $A_0 \cong B_0$ as $k$-algebras, and note that $A_0^{(\ell)}$ is a basic algebra of $A^{(\ell)}$. Therefore, $1 \leq mf(A_0) \leq frob(A_0) = mf(A) \leq frob(A)$ for any basic algebra $A_0$ of $A$. Recall that $A$ \textit{has an $\mathbb{F}_{\ell}$-form} if there is a $k$-vector space basis of $A$ such that all structure constants lie in $\mathbb{F}_{\ell}$. By \cite[Lemma 2.1]{K}, $A$ has an $\mathbb{F}_{\ell}$-form if and only if $A \cong A^{(\ell)}$ as $k$-algebras -- that is, if and only if $frob(A) = 1$.

\subsection{Results from Block Theory.}
\label{subsec:blocktheory}

Let $b$ be a block of $kG$. By this we mean that $b$ is a primitive idempotent in $Z(kG)$. We denote the Morita Frobenius and Frobenius numbers of $kGb$ by $mf(b)$ and $frob(b)$, respectively. Let $\sigma: k \rightarrow k$ be the Frobenius automorphism given by $\lambda \mapsto \lambda^\ell$ for all $\lambda \in k$. We also denote by $\sigma: kG \rightarrow kG$ the induced \textit{Galois conjugation} map on $kG$, defined by 
\[ \sigma \left(\sum_{g \in G} \alpha_g g \right) = \sum_{g \in G} \alpha_g^{\ell}g  \]
for all $\sum_{g \in G} \alpha_g g \in kG$. Although not an isomorphism of $k$-algebras, Galois conjugation is a ring isomorphism so it permutes the blocks of $kG$. We call $\sigma(b)$ (or $kG\sigma(b)$) the \textit{Galois conjugate} of $b$ (resp. $kGb$), and we say that two blocks $b$ and $c$ of $kG$ are \textit{Galois conjugate} if $b = \sigma^n(c)$ for some positive integer $n$. 

\begin{lemma}[Benson and Kessar {\cite{B/K}}]
\label{lem:galconj}
There is a $k$-algebra isomorphism $kGb^{(\ell)} \cong kG\sigma(b)$ between the first Frobenius twist of $kGb$ and the Galois conjugate of $kGb$. 
\end{lemma}

We fix an \textit{$\ell$-modular system} $(K, \mathcal{O}, k)$ with $K$ a field of characteristic 0 containing a $|G|$-th root of unity, $\nu: K \rightarrow \mathbb{Z} \cup \{ \infty \}$ a complete discrete valuation on $K$, $\mathcal{O}$ the valuation ring of $\nu$ with maximal ideal $\mathfrak{m}$, and $k$ the residue field $\mathcal{O}/\mathfrak{m}$. The canonical quotient map $\mathcal{O}G \rightarrow kG$ induces a bijection between the set of blocks of $\mathcal{O}G$ and the set of blocks of $kG$. If $b$ is a block of $kG$, we denote the corresponding block of $\mathcal{O}G$ by $\tilde{b}$. Blocks $\tilde{b}$ and $\tilde{c}$ of $\mathcal{O}G$ are said to be \textit{Galois conjugate} if $b$ and $c$ are Galois conjugate. 

Let Irr$_K(G)$ denote the set of $K$-valued irreducible characters of $G$ and let $e_{\chi}$ be the central idempotent of $KG$ corresponding to $\chi \in $ Irr$_K(G)$. Let Irr$_K(b) = \{ \chi \in \mbox{ Irr}_K(G) ~|~ \tilde{b}e_{\chi} = e_{\chi} \}$ denote the set of irreducible characters belonging to the block $b$. We fix an automorphism $\hat\sigma : K \rightarrow K$ such that $\hat\sigma(\zeta) = \zeta^{\ell}$ for any $\ell'$-root of unity $\zeta$ in $K$. Then $\hat\sigma$ induces an action on $KG$ via 
\[\hat\sigma\left(\sum_{g \in G} \alpha_g g\right)  = \sum_{g \in G} \hat\sigma(\alpha_g) g\]
for all $\sum_{g \in G} \alpha_g g \in KG$, and an action on Irr$_K(G)$ via 
\[^{\hat\sigma}\chi(g) = \hat\sigma(\chi(g))\]
for all $\chi \in $ Irr$_K(G)$ and all $g \in G$. Note that although $\hat\sigma$ may not preserve $\mathcal{O}$, it induces an action on the set of blocks compatible with the action of $\sigma$ on the blocks of $kG$. More precisely, we have the following.

\begin{lemma}
\label{lem:sigmahat} Let $b$ be a block of $kG$. Then
\begin{enumerate}[(a)]
\item$\hat{\sigma}(\tilde{b}) = \widetilde{\sigma(b)}$, and 
\item Irr$_K(\sigma(b)) = \{ ^{\hat\sigma}{\chi} ~|~ \chi \in \mbox{ Irr}_K(b)\}$.
\end{enumerate}
\end{lemma}

\begin{proof}
For part (a), see Kessar {\cite[Lemma 3.1]{K3}}. For part (b), we first note the following.
\begin{align*}
\hat\sigma (e_{\chi}) 	& = \hspace{1ex}  	\hat\sigma \left(\frac{ \chi(1) } {|G|} \sum_{g \in G} \chi(g^{-1}) g \right)\\
				& =  \hspace{1ex} 	\frac{\chi(1)}{|G|} \sum_{g \in G} \hat\sigma \left(\chi(g^{-1})\right) g\\
				& =  \hspace{1ex} 	\frac{^{\hat\sigma}\chi(1)}{|G|} \sum_{g \in G} {^{\hat\sigma}\chi}(g^{-1}) g\\
				& = \hspace{1ex}    e_{^{\hat\sigma}\chi}
\end{align*}
Suppose that $\chi \in $ Irr$_K(b)$. Then 
\[\widetilde{\sigma(b)} e_{^{\hat\sigma}\chi} = \hat\sigma(\tilde{b}) \hat\sigma(e_{\chi}) = \hat\sigma(\tilde{b} e_{\chi} ) =  \hat\sigma(e_{\chi}) = e_{^{\hat\sigma}\chi},\]
so $^{\hat\sigma}{\chi} \in $ Irr$_K(\sigma(b))$, showing that $ \{ ^{\hat\sigma}{\chi} ~|~ \chi \in \mbox{ Irr}_K(b)\} \subseteq $  Irr$_K(\sigma(b))$.

On the other hand, for any $\psi \in $ Irr$_K(\sigma(b))$, since $\hat{\sigma}$ is an automorphism of $K$ we can define a character $\chi \in $ Irr$_K(G)$ by $\chi(g) = {\hat{\sigma}}^{-1}\left( \psi(g)\right)$ for all $g \in G$, so $^{\hat\sigma}{\chi} = \psi$.  Since $\psi \in $ Irr$_K(\sigma(b))$, $\widetilde{\sigma(b)} e_{\psi} = e_{\psi}$, so 
\[\hat{\sigma}\left(\tilde{b}e_{\chi}\right) = \hat{\sigma} (\tilde{b}) \hat{\sigma} ( e_{\chi} ) = \widetilde{\sigma(b)} e_{^{\hat{\sigma}}\chi} = \widetilde{\sigma(b)} e_{\psi} = e_{\psi} = e_{^{\hat{\sigma}}\chi} = {\hat{\sigma}}(e_{\chi}).\] Therefore $\tilde{b}e_{\chi} = e_{\chi}$ so $\chi \in $ Irr$_K(b)$, hence Irr$_K(\sigma(b))  \subseteq  \{ ^{\hat\sigma}{\chi} ~|~ \chi \in \mbox{ Irr}_K(b)\}$ and the result follows.
\end{proof}

\begin{proposition}
\label{lem:mainlemma}
Let $b$ be a block of $kG$. Suppose that one of the following holds.
\begin{itemize}
	\item[(a)] $b \in \mathbb{Q}G$
	\item[(b)] There exist $\chi_1, \dots, \chi_r \in $ Irr$_K(b)$ for some $r \geq 1$ such that $\left(\chi_1 + \dots + \chi_r\right) (g) \in \mathbb{Q}$ for all $g \in G$
	\item[(c)] There exists $\chi \in $ Irr$_K(b)$ such that $\chi(1)_{\ell} = |G|_{\ell}$
	\item[(d)] The defect groups of $b$ are cyclic or dihedral
\end{itemize}
Then $mf(b) = 1$. 
\end{proposition}

\begin{proof}
If $b \in \mathbb{Q}G$ then $\sigma(b) = b$ since $\mathbb{Q}$ is stabilized by $\hat\sigma$. Therefore $kGb^{(\ell)} \cong kGb$ as $k$-algebras by Lemma~\ref{lem:galconj}, so $frob(b) = 1$ and therefore $mf(b) = 1$. 

Suppose that there exist $\chi_1, \dots , \chi_r \in $ Irr$_K(b)$ for some $r \geq 1$ such that $\left( \chi_1 + \dots + \chi_r \right) (g) \in \mathbb{Q}$ for all $g \in G$. Then $\left( {^{\hat\sigma}{\chi_1}} + \dots + {^{\hat\sigma}{\chi_r}} \right)(g)  = \hat\sigma \left( \chi_1 + \dots + \chi_r \right)(g) = \left( \chi_1 + \dots + \chi_r \right) (g)$ for all $g \in G$. It follows that $\{ {^{\hat\sigma}{\chi_1}} , \dots , {^{\hat\sigma}{\chi_r}}\}$ and $\{\chi_1 , \dots , \chi_r\}$ are equal as sets of irreducible characters, so $\sigma(b) = b$ by Lemma~\ref{lem:sigmahat} (b). Therefore $mf(b) = 1$ following the same argument as in part (a).

 By \cite[Theorem 6.1.1]{Block}, if there exists a $\chi \in $ Irr$_K(b)$ such that $\chi(1)_{\ell} = |G|_{\ell}$, then $kGb$ is a matrix algebra. Therefore $kGb$ has an $\mathbb{F}_{\ell}$-form for any $\ell$, so $mf(b) = 1$, showing part (c). If $b$ has cyclic defect then its basic algebras are Brauer tree algebras, so they are defined over $\mathbb{F}_{\ell}$. If $b$ has dihedral defect then its basic algebras are defined over $\mathbb{F}_2$ \cite{Erd}. Thus if $b$ has cyclic or dihedral defect then the Frobenius number of any basic algebra of $kGb$ is 1, so $mf(b) = 1$.
\end{proof}

\begin{lemma}
\label{lem:groupaut}
Let $b$ be a block of $kG$. Suppose that there exists a group automorphism $\varphi \in $Aut~$(G)$ such that for the induced $k$-algebra isomorphism $\varphi: kG \rightarrow kG$, $\varphi(b) = \sigma(b)$. Then $mf(b) = 1$. 
\end{lemma}

\begin{proof}
Since $\varphi|_{kGb}: kGb \rightarrow kG\sigma(b)$ is a $k$-algebra isomorphism, $kGb~\cong~kG\sigma(b)$. Therefore $kGb~\cong~kGb^{(\ell)}$ as $k$-algebras by Lemma~\ref{lem:galconj}, so $frob(b) = 1$, whence $mf(b) = 1$. 
\end{proof}

\begin{lemma}
\label{lem:twistedalg}
Let $G$ be a finite group such that $H^2(G, k^{\times}) \cong C_2$ and let $\gamma \in H^2(G, k^{\times})$. Then $mf(k_{\gamma}G)=1$. 
\end{lemma}

\begin{proof}
Define a map $\sigma:~H^2(G, k^{\times}) \rightarrow H^2(G, k^{\times})$ as follows. Let $\gamma \in H^2(G, k^{\times})$ and let $\tilde{\gamma}$ be a 2-cocycle representing $\gamma $. Then $\sigma(\gamma)$ is defined to be the class in $H^2(G, k^{\times})$ represented by the 2-cocycle given by
\[ (g, h) \mapsto \sigma(\tilde{\gamma}(g,h)),\]
for all $g, h \in G$. It is easy to check that $\sigma$ is a well-defined group homomorphism on $H^2(G, k^{\times}) $. If $\gamma$ is non-trivial then so is $\sigma(\gamma)$, so since $H^2(G, k^{\times}) \cong C_2$, $k_{\gamma}G \cong k_{\sigma(\gamma)}G$ as $k$-algebras. 

Recall that $k_{\gamma}G^{(\ell)} \cong k_{\gamma}G$ as rings but not necessarily as $k$-algebras, and that scalar multiplication in $k_{\gamma}G^{(\ell)}$ is given by $\lambda.x = \lambda^{\frac{1}{\ell}}x$ for all $\lambda \in k, x \in k_{\gamma}G$. Let $\varphi: k_{\sigma(\gamma)}G \rightarrow  k_{\gamma}G^{(\ell)}$ be the map defined by 
\[ \varphi \left( \sum_{g \in G} \alpha_g g\right) = \sum_{g \in G} \alpha_g^{\frac{1}{\ell}} g\] 
for all $ \sum_{g \in G} \alpha_g g \in  k_{\sigma(\gamma)}G$. This is a ring isomorphism, and
\[\varphi \left(\lambda \sum_{g \in G} \alpha_g g\right) =
\sum_{g \in G} \left( \lambda \alpha_g \right)^{\frac{1}{\ell}} g = 
\lambda^{\frac{1}{\ell}} \sum_{g \in G}  \alpha_g^{\frac{1}{\ell}} g  = 
\lambda . \varphi \left( \sum_{g \in G} \alpha_g g \right)\] 
for all $\lambda \in k$ and $ \sum_{g \in G} \alpha_g g \in  k_{\sigma(\gamma)}G$, so $\varphi$ is in fact an isomorphism of $k$-algebras. Therefore $k_{\gamma}G\cong~ k_{\sigma(\gamma)}G \cong k_{\gamma}G^{(\ell)}$ as $k$-algebras, so $frob(k_{\gamma}G)=1$, hence $mf(k_{\gamma}G) = 1$.
\end{proof}

\subsection{Dominating blocks.}
\label{subsec:domblocks}

Let $G$ be a finite group with normal subgroup $Z$. Let $\closure{G} = G/Z$ and let $\mu: G \rightarrow \closure{G}$ be the natural quotient map. Denote also by $\mu:  kG \rightarrow k\closure{G}$ the induced $k$-algebra homomorphism given by
\[ \mu \left(\sum_{g \in G} \alpha_g g \right) = \sum_{g \in G} \alpha_g \mu(g) \]
for all $\sum_{g \in G} \alpha_g g \in kG$. If $b$ is a block of $kG$, then $\mu(b) = \closure{b}_1 + \dots \closure{b}_r$ for some $r \geq 0$, where  $\closure{b}_i$ are block idempotents of $k\closure{G}$. Recall that if $r \neq 0$, then $b$ is said to \textit{dominate} the blocks $\closure{b}_i$ of $k\closure{G}$, for $1 \leq i \leq r$, and each block $\closure{b}$ of $k\closure{G}$ is dominated by a unique block of $kG$. By identifying $\chi \in $ Irr$_K(\closure{G})$ with $\chi \circ \mu \in $ Irr$_K(G)$, we can consider Irr$_K(\closure{G})$ as a subset of Irr$_K(G)$. See \cite[Ch. 5, Section 8.2]{N/T} for more details. 

\begin{lemma}
\label{lem:domblocks}
Let $b$ be a block of $kG$.
\begin{itemize}
	\item[(a)] $b$ dominates some block of $k\closure{G}$ if and only if $b$ covers the principal block of $kZ$
	\item[(b)] $b$ dominates a block $\overline{b}$ of $k\closure{G}$ if and only if $\sigma(b)$ dominates $\sigma(\closure{b})$
	\item[(c)] If $Z \leq Z(G)$ and $b$ dominates some block of $k\closure{G}$, then $b$ dominates a unique block of $k\closure{G}$
	\item[(d)] If $Z$ is an $\ell'$-group (not necessarily central) and $b$ dominates some block of $k\closure{G}$, then $b$ dominates a unique block $\closure{b}$ of $k\closure{G}$ and $kGb \cong k{\closure{G}}{\closure{b}}$ as $k$-algebras
\end{itemize}
\end{lemma}

\begin{proof}
Part (a) follows directly from \cite[Ch. 5 Lemma 8.6 (i)]{N/T}. For part (b), note that by \cite[Ch. 5, Lemma 8.6 (ii)]{N/T}, $b$ dominates $\closure{b}$ if and only if Irr$_K(\overline{b}) \subseteq$ Irr$_K(b)$, where we identify characters in Irr$_K(\closure{G})$ with characters in Irr$_K(G)$ as discussed above. Irr$_K(\overline{b}) \subseteq$ Irr$_K(b)$  if and only if we have the following. 
\[ \mbox{Irr}_K\left(\sigma\left(\closure{b}\right)\right) = 
\{ {^{\hat{\sigma}}\chi}~|~ \chi \in \mbox{Irr}_K\left(\closure{b}\right)\} \subseteq 
 \{ {^{\hat{\sigma}}\chi}~ |~\chi \in \mbox{Irr}_K(b)\}  = 
\mbox{ Irr}_K(\sigma(b))
\]
Therefore $b$ dominates $\closure{b}$ if and only if $\sigma(b)$ dominates $\sigma\left(\closure{b}\right)$.  

Part (c) follows from \cite[Ch. 5 Theorem 8.11]{N/T}. Finally for part (d), suppose that $Z$ is an $\ell'$-subgroup of $G$ and that $b$ dominates a block $\closure{b}$ of $k\closure{G}$. Then by \cite[Ch. 5, Theorem 8.8]{N/T}, $\closure{b}$ is the unique block of $k\closure{G}$ dominated by $b$, and Irr$_K(b)$ = Irr$_K\left(\closure{b}\right)$. Therefore $\mu(b) = \closure{b}$, so $\mu:kG \rightarrow k\closure{G}$ restricts to another surjection $\closure{\mu}: kGb \rightarrow k\closure{G}\closure{b}$ given by 
\[ \closure{\mu} \left( \left(\sum_{g \in G} \alpha_g g \right) b \right) = \left( \sum_{g \in G} \alpha_g \mu(g) \right) \closure{b} \]
for all $\sum_{g \in G} \alpha_g g \in kG$. It only remains to show that this is an injection. Since Irr$_K(b)$ = Irr$_K\left(\closure{b}\right)$ and rank$_{k}(kGb) = $ dim$_K(KGb)$,
\[
\mbox{rank}_{k}(kGb) 
=  \sum_{\chi \in \mbox{\footnotesize{Irr}}_K(b)} \chi(1)^2
=  \sum_{\chi \in \mbox{\footnotesize{Irr}}_K\left(\closure{b}\right)} \chi(1)^2  
= \mbox{ rank}_{k}(k\closure{G}\closure{b}),
\]
so $\closure{\mu}: kGb \rightarrow k\closure{G}\closure{b}$ is a $k$-algebra isomorphism as required for part (d). 
\end{proof}


\section{THE ALTERNATING GROUPS}
\label{sec:AltGps}

\begin{theorem}
\label{thrm:alt}
Let $G$ be $S_n$, $A_n$, or a double cover of $S_n$ or $A_n$, and let $b$ be a block of $kG$. Then $mf(b) = 1$.
\end{theorem}

\begin{proof}
The irreducible characters of $S_n$ are rational valued so the result follows immediately for blocks of $kS_n$ by Proposition~\ref{lem:mainlemma} (b). The irreducible characters of $A_n$ arise as restrictions of irreducible characters of $S_n$, which are parametrized by the partitions $\lambda$ of $n$. Suppose $b$ is the block of $kS_n$ containing the irreducible character $\chi_{\lambda}$ associated with a partition $\lambda$. By \cite[Lemma 12.1]{O2}, if $\lambda$ is symmetric then $\chi_{\lambda}|_{A_n}$ is an irreducible character of $A_n$, so $\chi_{\lambda}|_{A_n}$ is a rational valued character of $A_n$. If $\lambda$ is not symmetric then $\chi_{\lambda}|_{A_n} = \chi_{\lambda}^1 + \chi_{\lambda}^2$ is the sum of two irreducible conjugate characters of $A_n$, and these may not be rational valued. By \cite[Proposition 12.2]{O2}, if $b$ has non-trivial defect, then $\chi_{\lambda}^1$ and $\chi_{\lambda}^2$ appear in the same block of $kA_n$, and we note that their sum is rational valued. If $b$ has trivial defect then $\chi_{\lambda}^1$ and $\chi_{\lambda}^2$ are in separate blocks of $kA_n$ \cite[Theorem 6.1.46]{J/K}, each of defect zero. Therefore, any block of $kA_n$ satisfies the hypothesis of at least one of parts (a), (b) and (d) of Proposition~\ref{lem:mainlemma}, and therefore, the Morita Frobenius number of all blocks of $kA_n$ is 1.

Let $\widetilde{S}_n$ denote a double cover of the symmetric group.  When $\ell$ is odd, $S_n$ is a quotient of $\widetilde{S}_n$ by a central $\ell'$-subgroup, so by \cite[Ch. 5 Theorem 8.8]{N/T} $k\widetilde{S}_n$ has two types of blocks -- blocks which dominate unique blocks of $kS_n$, and blocks which do not dominate any block of $kS_n$. First, suppose $c$ is a block of $k\widetilde{S}_n$ which dominates a block $b$ of $kS_n$. Then $k\widetilde{S_n}c \cong kS_nb$ as $k$-algebras by Lemma~\ref{lem:domblocks} (d), so $mf(c) = mf(b) = 1$. 

Now suppose $c$ is a block of $k\widetilde{S}_n$ which does not dominate a block of $kS_n$. Then $c$ contains only spin characters and these are parametrized by the strict partitions of $n$ -- partitions of $n$ which have no repeated parts. The \emph{parity} of a partition is 
\begin{equation*}
  \epsilon(\lambda)=\begin{cases}
    				0 & \text{if ($n$ minus the number of parts in $\lambda$) is even},\\
    				1 & \text{otherwise}.
  			\end{cases}
\end{equation*}

If $\epsilon(\lambda)= 0$ then $\lambda$ has one associated spin character, $\chi_{\lambda}$. Then $\chi_{\lambda}(g) \neq 0$ only if $g$ has cycle type with all odd parts, and the character values can be calculated using an analogue of the Murnaghan Nakayama formula \cite{Mo}. In particular, $\chi_{\lambda}(g) \in \mathbb{Q}$ for all $g \in G$. If $\epsilon(\lambda)= 1$ then $\lambda$ has two associated spin characters, $\chi_{\lambda}$ and its \emph{associate} $\chi_{\lambda}^{a}$ and there are two possibilities to consider. Firstly, if $\lambda$ is equal to its $\ell$-bar core (see \cite[Definition 5]{Ca}) then $\chi_{\lambda}$ and $\chi_{\lambda}^{a}$ lie in $\ell$-blocks of defect zero. Secondly, if $\lambda$ is not equal to its $\ell$-bar core, then $\chi_{\lambda}$ and $\chi_{\lambda}^{a}$ appear in the same block and $\chi_{\lambda}(g) = - \chi_{\lambda}^{a}(g)$ for all $g \in \widetilde{S}_n$ \cite[Theorems A and B]{Ca}. Therefore $\left( \chi_{\lambda} + \chi_{\lambda}^{a} \right) (g)  \in \mathbb{Q}$ for all $g \in G$. Thus the result follows for all blocks $c$ of $k\widetilde{S}_n$ when $\ell$ is odd by Proposition~\ref{lem:mainlemma} (a), (b) and (d). 

When $\ell=2$, the 2-blocks of $k\widetilde{S}_n$ are in one-to-one correspondence with the 2-blocks of $kS_n$ \cite[Ch. 5, Theorem 8.11]{N/T}, so each block of $k\widetilde{S}_n$ contains at least one rational valued character of $S_n$. The result therefore follows for all 2-blocks of $k\widetilde{S}_n$ by Proposition~\ref{lem:mainlemma} (b).

Finally, let $\widetilde{A}_n$ denote a double cover of $A_n$. Suppose $d$ is a block of $k\widetilde{A}_n$ covered by a block $c$ of $k\widetilde{S}_n$. If $c$ has non-trivial defect, then by \cite[Proposition 3.16 (i)]{K2}, $d = c$ so $k\widetilde{A}_n d = k\widetilde{A}_n c$. By the arguments for $k\widetilde{S}_n$ above, $k\widetilde{S}_n c$ satisfies at least one of the hypotheses of parts (a), (b) and (d) of Proposition~\ref{lem:mainlemma} and therefore so does $k\widetilde{A}_n c$. It follows that $mf(k\widetilde{A}_n d) = mf(k\widetilde{A}_n c) = 1$. Now suppose that $c$ has trivial defect. Then $d$ also has trivial defect so $mf(d) = 1$ by Proposition~\ref{lem:mainlemma} (d).
\end{proof}

\section{FINITE GROUPS OF LIE TYPE IN DEFINING CHARACTERISTIC}
\label{sec:definingchar}

\begin{theorem}
\label{lem:autgalconj}
Let $\textbf{G}$ be a simple, simply-connected algebraic group defined over an algebraic closure of the field of $\ell$ elements. Let $q$ be a power of $\ell$ and let $F: \textbf{G} \rightarrow \textbf{G}$ be a Steinberg morphism with respect to an $\mathbb{F}_q$-structure with finite group of fixed points, $\textbf{G}^F$. Let $b$ be a block of of $k\textbf{G}^F$ with Galois conjugate $\sigma(b)$. Then there exists a group automorphism $\varphi:\textbf{G}^F \rightarrow \textbf{G}^F$ such that for the induced $k$-algebra isomorphism $\varphi: k\textbf{G}^F \rightarrow k\textbf{G}^F$, $\varphi(b) = \sigma(b)$.
\end{theorem}

\begin{proof}
By \cite[Theorems 8.3, 8.5]{H}, since $\textbf{G}$ is simply-connected and $\ell$ divides $q$, $k\textbf{G}^F$ has $|Z(\textbf{G}^F)| + 1$ blocks; one of trivial defect which contains the Steinberg character, and $|Z(\textbf{G}^F)|$ of full defect. Note that these results also hold for the Suzuki and Ree groups.

First suppose that $Z(\textbf{G}^F) \leq C_2$. Then $k\textbf{G}^F$ has at most three blocks. One of these blocks contains the trivial character and another contains the Steinberg character, so by the proof of Proposition~\ref{lem:mainlemma} (b), all blocks $b$ of $k\textbf{G}^F$ are stabilized by Galois conjugation. We can therefore let $\varphi: \textbf{G}^F \rightarrow \textbf{G}^F$ be the identity map.

Now suppose that $Z(\textbf{G}^F) \cong C_m$ for some $m>2$ coprime to $\ell$. Let $Z(\textbf{G}^F) = \langle g \rangle$. Then $Z(\textbf{G}^F)$ has $m$ irreducible characters $\chi_i :  Z(\textbf{G}^F) \rightarrow  K$, and to each character there is an associated central primitive idempotent $e_i$ of $KZ(\textbf{G}^F)$,
\[ e_i 	 =  \frac{1}{m} \sum_{0 \leq a \leq m-1} \chi_i(g^{a})g^{-a}\]
for $0 \leq i \leq m-1$. Since $m$ is coprime to $\ell$ it is invertible in $\mathcal{O}$, so $e_i \in \mathcal{O}\textbf{G}^F$. Let $\bar{e}_i$ be the image of $e_i$ in $k\textbf{G}^F$ under the canonical quotient mapping $\mathcal{O}\textbf{G}^F \rightarrow k\textbf{G}^F$, 
\[ \bar{e}_i 	 =  \frac{1}{m} \sum_{0 \leq a \leq m-1} \overline{\chi_i(g^{a})} g^{-a}.\]
Then $\bar{e}_i$ is a block of $kZ(\textbf{G}^F)$ and is a central, but not necessarily primitive, idempotent of $k\textbf{G}^F$. Since $k\textbf{G}^F$ has $m+1$ blocks, there are exactly $m+1$ primitive central idempotents in $k\textbf{G}^F$. Clearly, the blocks of $kZ(\textbf{G}^F)$ are $\textbf{G}^F$-stable. Therefore, precisely one $\bar{e}_i$ is imprimitive in $k\textbf{G}^F$. Since the trivial and Steinberg characters of $\textbf{G}^F$ both restrict to the trivial character on $Z(\textbf{G}^F)$, it follows that the principal block of $kZ(\textbf{G}^F)$ is imprimitive in $k\textbf{G}^F$ and splits into the principal and Steinberg blocks of $k\textbf{G}^F$. Galois conjugation stabilizes the principal and Steinberg blocks, as discussed above, so it only remains to consider the $m-1$ blocks of $k\textbf{G}^F$ with block idempotent $\bar{e}_i$.

Galois conjugation acts on $\bar{e}_i$ by
\[ \sigma(\bar{e}_i) = \frac{1}{m} \sum_{0 \leq a \leq m-1} \overline{\chi_i(g^{a})}^{\ell} g^{-a}.\]
The action of $\sigma$ is trivial if $\ell \equiv 1$ mod m, so from now on we assume that $\ell \nequiv 1$ mod $m$. 

Let $F_{\ell}:\textbf{G} \rightarrow \textbf{G}$ be the $\mathbb{F}_{\ell}$-split Steinberg endomorphism given in \cite[Example 22.6]{M/T} which commutes with the Steinberg morphism $F$. Suppose $g \in \textbf{G}^F$. Then $F(F_{\ell}(g)) = F_{\ell}(F(g)) = F_{\ell}(g)$, so $F_{\ell}(g) \in \textbf{G}^F$ for all $g \in \textbf{G}^F$. Since $F_{\ell}$ is injective, it follows that $F_{\ell}(\textbf{G}^F) = \textbf{G}^F$. Therefore, $F_{\ell}$ is an automorphism of $\textbf{G}^F$ and so it restricts to an automorphism of $Z(\textbf{G}^F)$. We claim that $F_{\ell}(z) = z^{\ell}$ for all $z \in Z(\textbf{G}^F)$. 

Suppose that $\textbf{G}^F = SL_n(q)$ or $SU_n(q)$. Then if $(\alpha_{ij}) \in Z(\textbf{G}^F$), $F_{\ell}(\alpha_{ij}) = (\alpha_{ij}^{\ell}) = (\alpha_{ij})^{\ell}$. Next, suppose that $\textbf{G}^F =$ Spin$_{2n}^{+}(q)$ with $n \geq 5$ odd and $4|q-1$, so $Z(\textbf{G}^F) \cong C_4$. Since we are assuming that $\ell \nequiv$ 1 mod $m$, this only occurs if $\ell \equiv 3$ mod 4. Suppose that $F_{\ell}|_{Z(\textbf{G}^F)}$ is the trivial automorphism of $C_4$. Then $Z(\textbf{G}^F)$ is central in the fixed points of $F_{\ell}$, Spin$_{2n}^+({\ell})$. But $Z($Spin$_{2n}^+({\ell})) \cong C_2$, so this is impossible. Therefore $F_{\ell}|_{Z(\textbf{G}^F)}$ is the non-trivial automorphism of $C_4$, so $F_{\ell}(z) = z^3 = z^{\ell}$ for every $z \in Z(\textbf{G}^F) $. 

Now, suppose that $\textbf{G}^F =E_6(q)$ or $^2E_6(q)$ and $Z(\textbf{G}^F) \cong C_3$, so $\ell \equiv 2$ mod $3$. If $F_{\ell}|_{Z(\textbf{G}^F)}$ is the trivial automorphism of $C_3$ then $Z(\textbf{G}^F)$ is central in the fixed points of $F_{\ell}$, $E_6(\ell)$. However, $Z(E_6(\ell))$ is trivial when $\ell \equiv 2$ mod $3$, so again we get a contradiction. Therefore $F_{\ell}|_{Z(\textbf{G}^F)}$ is the non-trivial automorphism of $C_3$, so $F_{\ell}(z) = z^2 = z^{\ell}$ for every $z \in Z(\textbf{G}^F) $. This shows the claim for all $\textbf{G}^F$ such that $Z(\textbf{G}^F) \cong C_m$ when $\ell \nequiv 1$ mod $m$.

The automorphism $F_{\ell}$ therefore induces an action $\bar{e}_i$ as follows.
\begin{align*}
F_{\ell}(\bar{e}_i) 	& = \hspace{1ex}   \frac{1}{m} \sum_{0 \leq a \leq m-1} \overline{\chi_i(g^a)} F_{\ell}(g^{-a}) \\
			& = \hspace{1ex}   \frac{1}{m} \sum_{0 \leq a \leq m-1} \overline{\chi_i(g^a)} g^{-\ell a}
\end{align*}
Let $\varphi = F_{\ell} ^{\phi(m)-1}$, where $\phi$ is the Euler totient function, and let $\omega$ be a primitive $m$-th root of unity such that $\chi_i(g^a) = \omega^{ia}$ for $1 \leq a \leq m$. Then
\begin{align*}
\varphi(\bar{e}_i) 	& = \hspace{1ex}    \frac{1}{m} \sum_{0 \leq a \leq m-1} (\overline{\omega^{ia}}) g^{-\ell^{\phi(m)-1}a} \\
		& = \hspace{1ex}    \frac{1}{m} \sum_{0 \leq a' \leq m-1} (\overline{\omega^{i \ell a'}}) g^{-a'},
\end{align*}
letting $a' =\ell^{\phi(m)-1}a$ so that $ \ell a' =\ell^{\phi(m)}a \equiv a $ mod $m$. Therefore 
\[ \varphi(\bar{e}_i) =\frac{1}{m} \sum_{0 \leq a' \leq m-1} \overline{\chi_i(g^{a'})} g^{-a'} = \sigma(\bar{e}_i).\] 
This shows the result for all $\textbf{G}^F$ such that $Z(\textbf{G}^F) \cong C_m$, $m > 2$ and $m$ is coprime to $\ell$. 

Finally, suppose that $\textbf{G}^F = $ Spin$_{2n}^+ (q)$, with $n \geq 4$ even and $\ell$ odd, so $Z(\textbf{G}^F) \cong C_2 \times C_2$. The irreducible characters of $C_2 \times C_2$ are rational valued so the associated central primitive idempotents of $kZ(\textbf{G}^F)$ are stabilized by Galois conjugation. It follows that the central primitive idempotents of $k\textbf{G}^F$ are also stabilized by Galois conjugation, so again, we can let $\varphi$ be the identity map. 
\end{proof}

\begin{corollary}
\label{corr:defining}
Let $k\textbf{G}^F$ be as in Theorem~\ref{lem:autgalconj}. Then,
\begin{enumerate}[(a)]
	\item For any block $b$ of $k\textbf{G}^F$, $mf(b) = 1$, and 
	\item If $Z$ is a non-trivial central subgroup of $\textbf{G}^F$ and $\overline{b}$ is a block of $k(\textbf{G}^F/Z)$, then $mf\left(\overline{b}\right) = 1$.
\end{enumerate}
\end{corollary}

\begin{proof}
Part (a) follows from Theorem~\ref{lem:autgalconj} and Lemma~\ref{lem:groupaut}. For part (b), suppose that $\overline{b}$ is a block of $k(\textbf{G}^F/Z)$ dominated by a block $b$ of $k\textbf{G}^F$. Then by part (a), $mf(b) = 1$. As we are in defining characteristic, $Z(\textbf{G}^F)$ is an $\ell'$-group, so it follows from Lemma~\ref{lem:domblocks} (d) that $k\textbf{G}^Fb \cong k(\textbf{G}^F/Z)\overline{b}$ as $k$-algebras. Therefore $mf\left(\overline{b}\right)= mf(b) = 1$. 
\end{proof}

\section{UNIPOTENT BLOCKS OF FINITE GROUPS OF LIE TYPE IN NON-DEFINING CHARACTERISTIC}
\label{sec:unip}

In Section 5 we continue to assume that $k = \overline{\mathbb{F}}_{\ell}$, an algebraic closure of the field of $\ell$ elements. Let $p$ be a prime different to $\ell$, and let $\textbf{G}$ be a simple simply-connected algebraic group defined over an algebraic closure of the field of $p$ elements. Fix $q$, a power of $p$, and let $F: \textbf{G} \rightarrow \textbf{G}$ be the Frobenius morphism with respect to an $\mathbb{F}_q$-structure. Let $\textbf{G}^F$ be the fixed points of $\textbf{G}$ under $F$ -- a finite group of Lie type in non-defining characteristic. First we recall some standard notions from $e$-Harish Chandra Theory. See \cite{D/M} and \cite{B/M/M} for more details. 

\subsection{e-Harish Chandra Theory and Unipotent Blocks.}
\label{subsec:dHarChandTheory}

We denote by $P_{\left(\textbf{G}, F\right)}(x)$ the polynomial order of $\textbf{G}^F$; i.e. $P_{\left(\textbf{G}, F\right)}(x)$ is the unique polynomial such that $P_{(\textbf{G}, F)}(q^m) = \left|\textbf{G}^F\right|^m$ for infinitely many $m \in \mathbb{N}$. An $F$-stable torus $\textbf{T}$ is called a \textit{$e$-torus} if $P_{(\textbf{T},F)}(x)$ is a power of the $e$-th cyclotomic polynomial, $\Phi_e$, where $e$ is some natural number. An \textit{$e$-split Levi subgroup} \textbf{L} of $\textbf{G}$ is the centralizer in $\textbf{G}$ of some $e$-torus of $\textbf{G}$. Recall that for an $F$-stable Levi subgroup $\textbf{L}$ in an $F$-stable parabolic $\textbf{P}$ of $\textbf{G}$, there exist linear maps \textit{Deligne Lusztig induction} and \textit{restriction} given by 
\[ R_{\textbf{L} \subset \textbf{P}}^{\textbf{G}}: \mathbb{Z} \mbox{Irr}_K(\textbf{L}^F) \rightarrow \mathbb{Z} \mbox{Irr}_K(\mathbf{G}^F)\]
\[ ^*R_{\textbf{L} \subset \textbf{P}}^{\textbf{G}}: \mathbb{Z} \mbox{Irr}_K(\textbf{G}^F) \rightarrow \mathbb{Z} \mbox{Irr}_K(\mathbf{L}^F).\] 

An irreducible character $\chi$ of $\textbf{G}^F$ is called \textit{unipotent} if there exists an $F$-stable torus $\textbf{T}$ such that $\chi$ is a constituent of $R_{\textbf{T}}^{\textbf{G}}(1)$. The set of unipotent characters of $\textbf{G}^F$ is denoted by $\mathcal{E}(\textbf{G}^F, 1)$. Although it is not known in general whether $R_{\textbf{L} \subset \textbf{P}}^{\textbf{G}}$ and $^*R_{\textbf{L} \subset \textbf{P}}^{\textbf{G}}$ are independent of the choice of $\textbf{P}$, they are known to be independent for unipotent characters \cite{B/Mi}. We will therefore drop the reference to $\textbf{P}$ and denote Deligne Lusztig induction and restriction by $R_\mathbf{L}^\textbf{G}$ and $^*R_\mathbf{L}^\textbf{G}$ respectively. An $\ell$-block of $\textbf{G}^F$ is \textit{unipotent} if it contains a unipotent character. An irreducible character $\chi$ of $\textbf{G}^F$ is \textit{$e$-cuspidal} if $ ^*R_{\textbf{L}}^{\textbf{G}}(\chi) = 0$ for all proper $e$-split Levi subgroups $\textbf{L}$ of $\textbf{G}$. Note that \textit{cuspidal} is widely used instead of 1-cuspidal. 

A pair $(\textbf{L}, \lambda)$ is called \textit{unipotent $e$-split} if $\textbf{L}$ is an $e$-split Levi of $\textbf{G}$ and $\lambda$ is a unipotent character of $\textbf{L}^F$. If $\lambda$ is also $e$-cuspidal, then $(\textbf{L}, \lambda)$ is called a \textit{unipotent $e$-cuspidal pair}. The \textit{e-Harish Chandra series} above a unipotent $e$-cuspidal pair $(\mathbf{L},\lambda)$ is the set of unipotent characters 
\[ \mbox{Irr}_K (\textbf{G}^F, (\textbf{L}, \lambda)) = \{\gamma \in \mathcal{E}(\textbf{G}^F, 1) : \gamma \mbox{ is an irreducible constituent of } R_\mathbf{L}^\textbf{G}(\lambda)\}.\]
The set of irreducible unipotent characters of $\textbf{G}^F$ is partitioned by the $e$-Harish Chandra series of $\textbf{G}^F$-conjugacy classes of unipotent $e$-cuspidal pairs \cite[Theorem 7.5 (a)]{B/M}:
\[ \mbox{Irr}_K (\textbf{G}^F) = \dot{\bigcup} \mbox{ Irr}_K (\textbf{G}^F, (\textbf{L}, \lambda)),\]
where the $(\textbf{L}, \lambda)$ run over a system of representatives of $\textbf{G}^F$-conjugacy classes of unipotent $e$-cuspidal pairs of $\textbf{G}$. A unipotent $e$-cuspidal pair $(\textbf{L}, \lambda)$ is said to have \textit{$\ell$-central defect} if \\${\lambda(1)_{\ell}|Z(\textbf{L})^F|_{\ell}=|\textbf{L}^F|_{\ell}}$. If $\ell$ is odd, good for $\textbf{G}$ (see \cite[Section 1.1]{C/E}), and $\ell \neq 3$ if $^3D_4$ is involved in $\textbf{G}$, then all unipotent $e$-cuspidal pairs of $\textbf{G}$ are of $\ell$-central defect \cite[Proposition 4.3]{C/E}. 

We define $e_{\ell}(q)$ to be the order of $q$ modulo $\ell$ if $\ell > 2$, and the order of $q$ modulo 4 if $\ell =2$. 

\pagebreak
\begin{theorem}
\label{thrm:key}
Let $e = e_{\ell}(q)$. 
\begin{enumerate}[(a)]
	\item Let $(\textbf{L}, \lambda)$ be a unipotent $e$-cuspidal pair of $\textbf{G}$. Then all irreducible constituents of $R_{\mathbf{L}}^{\mathbf{G}}(\lambda)$ lie in the same $\ell$-block, $b_{\textbf{G}^F}(\textbf{L}, \lambda)$, of $\textbf{G}^F$.
	\item There exists a surjection 
		\begin{align*}
		\left\{ \begin{array}{c}
				\textbf{G}^F \mbox{-conjugacy classes of} \\
				\mbox{unipotent }e \mbox{-cuspidal } \\
				\mbox{pairs of } \textbf{G} 
		\end{array} \right\} & \twoheadrightarrow  \left\{ \begin{array}{c} Unipotent  \\ \ell \mbox{-blocks of }\textbf{G}^F \end{array} \right\}\\
		(\textbf{L}, \lambda) 								& \mapsto 		     b_{\textbf{G}^F}(\textbf{L}, \lambda)
		\end{align*}
		where $(\textbf{L}, \lambda)$ is a representative of a $\textbf{G}^F$-conjugacy class of unipotent $e$-cuspidal pairs of $\textbf{G}$ and $b_{\textbf{G}^F}(\textbf{L}, \lambda)$ is the $\ell$-block of $\textbf{G}^F$ containing all irreducible components of $R_{\mathbf{L}}^{\mathbf{G}}(\lambda)$.
	\item The surjection in (b) restricts to a bijection if we only consider unipotent $e$-cuspidal pairs of central $\ell$-defect.
		\begin{align*}
				\left\{ \begin{array}{c}
				\textbf{G}^F \mbox{-conjugacy classes of}\\
				\mbox{unipotent }e \mbox{-cuspidal pairs of } \textbf{G}\\
				\mbox{of } \ell \mbox{-central defect}
		\end{array} \right\} & \leftrightarrow   \left\{  \begin{array}{c} Unipotent  \\ \ell \mbox{-blocks of }\textbf{G}^F \end{array}  \right\}\\
		(\textbf{L}, \lambda) 									& \mapsto 		     b_{\textbf{G}^F}(\textbf{L}, \lambda)
		\end{align*}
		In particular, when $\ell$ is odd, good for $\textbf{G}$ and $\ell \neq 3$ if $^3D_4$ is involved in $\textbf{G}$, then the surjection from part (b) is itself a bijection. 
	\item If $\ell$ is odd or $\textbf{G}$ is of exceptional type, then the $\ell$-block $b_{\textbf{G}^F}(\textbf{L}, \lambda)$ has a defect group $P$ such that $Z(\textbf{L})^F_{\ell} \trianglelefteq P$ and $P / Z(\textbf{L})^F_{\ell}$ is isomorphic to a Sylow $\ell$-subgroup of $W_{\textbf{G}^F}(\textbf{L}, \lambda)$. 
\end{enumerate}
\end{theorem}

\begin{proof}
Parts (a), (b) and (c) were proved by Enguehard \cite[Theorem A]{E}. It therefore only remains to show part (d). By the proof of \cite[Theorem 7.12]{K/M}, for an $\ell$-block $b=b_{\textbf{G}^F}(\textbf{L}, \lambda)$, we have the following inclusion of Brauer pairs of $\textbf{G}^F$
\[ \left(\{1\}, b \right) \trianglelefteq \left(Z(\textbf{L})^F_{\ell}, b_{\textbf{L}^F}(\lambda)\right)  \trianglelefteq \left(P, e_{b}\right),\]
where $b_{\textbf{L}^F}(\lambda)$ the block of $k\textbf{L}^F$ containing $\lambda$, $e_{b}$ is a block of $C_{\textbf{G}^F}(P)$, $ \left(Z(\textbf{L})^F_{\ell}, b_{\textbf{L}^F}(\lambda)\right)$ is self-centralizing and $\left(P, e_{b}\right)$ is maximal. By \cite[Lemma 2.1]{K/M}, $  P / \left(P \cap Z(\textbf{L})^F_{\ell} \right) = P / Z(\textbf{L})^F_{\ell} $ is isomorphic to a Sylow $\ell$-subgroup of 
\[N_{\textbf{G}^F}\left(Z(\textbf{L})^F_{\ell}, b_{\textbf{L}^F}(\lambda)\right) \big/ C_{\textbf{G}^F}\left(Z(\textbf{L})^F_{\ell}\right).\]
Since $ C_{\textbf{G}^F}\left(Z(\textbf{L})^F_{\ell}\right) = \textbf{L}^F$ (see the proof of \cite[Theorem 7.12]{K/M2}), $P / Z(\textbf{L})^F_{\ell}$ is therefore isomorphic to a Sylow $\ell$-subgroup of 
\[N_{\textbf{G}^F}\left(Z(\textbf{L})^F_{\ell}, b_{\textbf{L}^F}(\lambda)\right) \big/ \textbf{L}^F = W_{\textbf{G}^F}(\textbf{L}, \lambda),\]
as required.
\end{proof}

\begin{lemma}
\label{lem:lambdarat}
Let $(\textbf{L}, \lambda)$ be a unipotent $e$-cuspidal pair of $\textbf{G}$ and suppose that $\lambda$ is rational valued. Then $mf\left(b_{\textbf{G}^F}(\textbf{L}, \lambda)\right)~=~1$.
\end{lemma}

\begin{proof}
Let $b = b_{\textbf{G}^F}(\textbf{L}, \lambda) $ and assume that $\lambda$ is rational valued so $^{\hat{\sigma}}{\lambda} = \lambda$ (see Section~\ref{subsec:blocktheory}).  By the Deligne Lusztig induction character formula \cite[Proposition 12.2]{D/M}, $^{\hat{\sigma}}\left( R_{\textbf{L}}^{\textbf{G}}(\lambda)  \right) = R_{\textbf{L}}^{\textbf{G}}\left({^{\hat{\sigma}}{\lambda}}\right)$. Suppose $\chi \in R_{\textbf{L}}^{\textbf{G}}(\lambda) \subseteq$~Irr$_K(b)$. Then ${^{\hat{\sigma}}{\chi}} \in R_{\textbf{L}}^{\textbf{G}}({^{\hat{\sigma}}{\lambda}}) = R_{\textbf{L}}^{\textbf{G}}({\lambda})$, so ${^{\hat{\sigma}}{\chi}} \in $ Irr$_K(b)$. By Lemma~\ref{lem:sigmahat}~(b), Irr$_K(\sigma(b)) = \{ ^{\hat\sigma}{\chi} ~|~ \chi \in \mbox{ Irr}_K(b)\}$, so it follows that $\sigma(b) = b$. Therefore $k\textbf{G}^Fb \cong k\textbf{G}^Fb ^{(\ell)}$ as $k$-algebras by Lemma~\ref{lem:galconj}, so $frob(b) = 1$, hence $mf(b) = 1$.
\end{proof}

\begin{lemma}
\label{lem:cuspchars}
Let $b$ be a block of $k\textbf{G}^F$ containing an $e$-cuspidal unipotent character $\lambda$ of central $\ell$-defect. Suppose that $Z\left([\textbf{G}, \textbf{G}]^F\right)$ is an $\ell'$-group. Then all characters in Irr$_K(b)$ are $e$-cuspidal.
\end{lemma}

\begin{proof}
Let $\lambda_0 = \lambda|_{[\textbf{G}, \textbf{G}]^F}$. Because $\lambda$ is unipotent, results of Lusztig show that $\lambda_0$ is irreducible (see for example \cite[Proposition 3]{C/E2}). Since $\lambda$ is of central $\ell$-defect, $\left|\textbf{G}^F\right|_{\ell}= \lambda(1)_{\ell} \left|Z(\textbf{G}^F)\right|_{\ell} = \lambda_0(1)_{\ell} \left|Z(\textbf{G}^F)\right|_{\ell}$. As we are assuming that $Z\left([\textbf{G}, \textbf{G}]^F\right)$ is an $\ell'$-group, it follows from $\left|\textbf{G}^F\right| = \left| Z^{\circ}(\textbf{G})^F\right| \left|[\textbf{G}, \textbf{G}]^F\right|$ that $\lambda_0(1)_{\ell} = \left|[\textbf{G}, \textbf{G}]^F\right|_{\ell}$. Therefore $\lambda_0$ is in a block $\bar{b}$  of $[\textbf{G}, \textbf{G}]^F$ of defect 0.

Let $\theta \in $ Irr$_K(b)$. Since $b$ covers $\bar{b}$ and $\lambda_0$ is the only character in $\bar{b}$, $\theta$ covers $\lambda_0$. By \cite[Corollary 11.7]{C/R}, therefore $\theta = \omega \lambda$ for a uniquely determined character $\omega$ of $\textbf{G}^F /  [\textbf{G}, \textbf{G}]^F$. Since $ [\textbf{G}^F, \textbf{G}^F] \subseteq [\textbf{G}, \textbf{G}]^F$, $\textbf{G}^F /  [\textbf{G}, \textbf{G}]^F$ is abelian, so $\omega$ is a linear character.

As $\lambda$ is $e$-cuspidal, $\langle \lambda , R_{\textbf{M}}^{\textbf{L}} (\tau) \rangle~=~0$ for any proper $e$-split Levi subgroup $\textbf{M}$ of $\textbf{L}$ and for all $\tau \in $ Irr$_K\textbf{M}^F$. Because $\omega$ is linear, it follows that $\langle \omega \lambda , \omega R_{\textbf{M}}^{\textbf{L}} (\tau) \rangle = \langle \theta , R_{\textbf{M}}^{\textbf{L}} (\omega \tau) \rangle = 0$ for all $\tau \in $ Irr$_K\textbf{M}^F$. Let $\tilde{\tau} = \omega \tau$. Then $\tilde{\tau}$ runs over Irr$_K\textbf{M}^F$ as $\tau$ does, so $\langle \theta , R_{\textbf{M}}^{\textbf{L}} (\tilde{\tau}) \rangle = 0$ for all $\tilde{\tau} \in $ Irr$_K\textbf{M}^F$. Therefore $\theta$ is $e$-cuspidal, as required. 
\end{proof}


\subsection{A result of Puig.}

Theorem~\ref{thrm:Puig} shows that under certain conditions, Puig's result \cite[Theorem 5.5]{P} can be applied to a block $b = b_{\textbf{G}^F}(\textbf{L}, \lambda)$ to show that $\mathcal{O}\textbf{G}^Fb$ is Morita equivalent to a specific block of $\mathcal{O}N_{\textbf{G}^F}(\textbf{L}, \lambda)$. This result will be used later to calculate the Morita Frobenius number of some unipotent blocks of $E_8(q)$. 

First we recall the following. Suppose $M$ is a finite group with a normal $\ell'$-subgroup $U$, and suppose that $L \cong M/U$. Let $\mu: M \rightarrow L$ be the quotient map. If $d$ is the principal block of $\mathcal{O}U$, then Fong Reduction \cite[Theorem 6.8.7]{Block} yields the following inverse $\mathcal{O}$-algebra isomorphisms,
\begin{align*}
\mathcal{O}L~&  \tilde{\longrightarrow}~\mathcal{O}Md \\
x~& \longmapsto ~xd \\
\mu(y)~& \longmapsfrom~y,
\end{align*}
for all $x \in \mathcal{O}L$, $y \in \mathcal{O}Md$. 

\begin{theorem}
\label{thrm:Puig}
Suppose $\ell \geq 5$ and $\ell| q-1$. Let $(\textbf{L}, \lambda)$ be a proper unipotent $1$-cuspidal pair of $\textbf{G}$ of central $\ell$-defect. Let $b = b_{\textbf{G}^F}(\textbf{L}, \lambda)$ and suppose that $P=Z(\textbf{L})^F_{\ell}$ is a defect group of $b$.  Let $f = b_{\textbf{L}^F}(\lambda)$ be the block of $\mathcal{O}\textbf{L}^F$ containing $\lambda$. Then $f$ is a block of $\mathcal{O}N_{\textbf{G}^F}(\textbf{L}, \lambda)$ with defect group $P$ such that $\mathcal{O}\textbf{G}^Fb$ and $\mathcal{O}N_{\textbf{G}^F}(\textbf{L}, \lambda)f$ are Morita equivalent. 
\end{theorem}

\begin{proof} 
Since $\textbf{L}$ is a 1-split Levi subgroup, $\textbf{L}$ is contained in an $F$-stable parabolic subgroup of $\textbf{G}$, $\textbf{M}$, say. Let $\textbf{U}$ be the unipotent radical of $\textbf{M}$, so $\textbf{M} = \textbf{U} \rtimes \textbf{L}$. Set $\textbf{M}^F = \textbf{U}^F \rtimes \textbf{L}^F$. Then $\textbf{M}^F / \textbf{U}^F \cong \textbf{L}^F$. Let $\mu : \textbf{M}^F \rightarrow \textbf{L}^F$ be the quotient map. Let $N = N_{\textbf{G}^F}(\textbf{L}, \lambda)$ and let $c$ be the block of $k\textbf{M}^F$ that dominates $f$. We show that the hypotheses of \cite[Theorem 5.5]{P} are satisfied by $\textbf{M}^F$, $N$, $\textbf{L}^F$, $c$ and $f$. 

Because $\textbf{U}^F$ is an $\ell'$-group, $c$ dominates a unique block of $\mathcal{O}\textbf{L}^F$ by \cite[Ch. 5 Theorem 8.8]{N/T}, so $\mu(c) = f$. Let $d$ be the principal block of $\mathcal{O}\textbf{U}^F$. Then it follows from the isomorphisms due to Fong Reduction mentioned above, that $c = fd$. Since $d$ is central in $\mathcal{O}M$, therefore $cf = c$. Since $\lambda$ is a 1-cuspidal unipotent character in $f$ with central $\ell$-defect, Lemma~\ref{lem:cuspchars} shows that all the characters in $f$ are 1-cuspidal. It then follows by arguments given in \cite[5.3]{P} that $c(\mathcal{O}\textbf{G}^F)c = c(\mathcal{O}N)c$. 

Next, since $N_{\textbf{G}^F}(\textbf{L}, \lambda) \subseteq N_{\textbf{G}^F}(\textbf{L}^F, \lambda)$, $N$ normalizes $\textbf{L}^F$ and therefore $f$. By the proof of \cite[Corollary 1.18]{D/M}, $N_{\textbf{G}}(\textbf{L}) \cap \textbf{U} = \{1\}$. Therefore $N_{\textbf{M}^F}(\textbf{L}, \lambda) \cap \textbf{U}^F = \{1\}$, so $N_{\textbf{M}^F}(\textbf{L}, \lambda) \subseteq \textbf{L}^F $ and thus $\textbf{L}^F =  N_{\textbf{M}^F}(\textbf{L}, \lambda) = N \cap \textbf{M}^F$. By \cite[Proposition 2.2 (ii)]{C/E}, since $\ell \geq 5$ and $\textbf{L}$ is a proper Levi subgroup of $\textbf{G}$, $\textbf{L}^F = C_{\textbf{G}^F}\left(Z(\textbf{L})^F_{\ell}\right) = C_{\textbf{G}^F}(P)$. Therefore $f$ is a block of $\mathcal{O}C_{\textbf{G}^F}(P)$, so $Br_P(f) = f$. It follows that $Br_P(c) = Br_P(df) = \frac{1}{| \textbf{U}^F |}Br_P(f) \neq 0$, so all hypotheses of  \cite[Theorem 5.5]{P} are satisfied. 

Recall that we have the following inclusion of Brauer pairs $(1, b) \subseteq (P, f)$ from the proof of Theorem~\ref{thrm:key} (d). Therefore $Br_P(b)f = f$. Since $N/\textbf{L}^F$, the relative Weyl group of $(\textbf{L}^F, \lambda)$ in $\textbf{G}^F$, is an $\ell'$-group, \cite[5.5.4]{P} implies that $f$ is a block of $\mathcal{O}N_{\textbf{G}^F}(\textbf{L}, \lambda)$ with defect $P$, and $\mathcal{O}Nf$ and $\mathcal{O}\textbf{G}^Fb$ are source algebra equivalent, and hence Morita equivalent.
\end{proof}

\subsection{Unipotent blocks of finite groups of Lie type in non-defining characteristic.}
\label{subsec:proof}

\begin{theorem}
\label{thrm:MainTheorem}
Let $\textbf{G}$ be a simple, simply-connected algebraic group defined over an algebraic closure of the field of $p$ elements. Let $q$ be a power of $p$ and let $F: \textbf{G} \rightarrow \textbf{G}$ be the Frobenius morphism with respect to an $\mathbb{F}_q$-structure. Let $k$ be a field of postitive characteristic $\ell \neq p$ and let $e = e_{\ell}(q)$. Let $b$ be a unipotent $\ell$-block of $\textbf{G}^F$. Then 
\begin{enumerate} [(a)]
\item $mf(b) \leq 2$ and 
\item $mf(b) = 1$, except possibly when $b = b_{\textbf{G}^F}(\textbf{L}, \lambda)$ in one of the following situations. 
	\begin{itemize}
		\item $\textbf{G} = E_8$, $\textbf{L} = \phi_1^2.E_6$, $\lambda = E_6[\theta^i]$ $ (i=1,2)$, with $\ell=2$ and $e=1$
		\item $\textbf{G} = E_8$, $\textbf{L} = {\phi_2^2}.{^2E}_6$, $\lambda = {^2E_6}[\theta^i]$ $ (i=1,2)$, with $\ell \equiv 2$ mod 3 and $e=2$
	\end{itemize}
\end{enumerate}
\end{theorem}

\begin{proof}
Let $b = b_{\textbf{G}^F}(\textbf{L}, \lambda)$ be the block of $\textbf{G}^F$ containing all irreducible constituents of $R_{\mathbf{L}}^{\mathbf{G}}(\lambda)$, where $(\textbf{L}, \lambda)$ is a unipotent $e$-cuspidal pair of $\textbf{G}$ of central $\ell$-defect, as discussed in Theorem~\ref{thrm:key}. By \cite[Proposition 5.6 and Table 1]{G}, the unipotent characters of classical finite groups of Lie type (including $^3D_4(q)$) are rational valued, so by Proposition~\ref{lem:mainlemma} (b) and Lemma~\ref{lem:lambdarat} we need only consider the cases where $\textbf{G}$ is of exceptional type, $\textbf{L}$ contains some component of exceptional type, and $\lambda$ is not rational valued. These $e$-cuspidal pairs can be identified using \cite[Appendix: Table 1]{B/M/M}, \cite{E} and \cite[Chapter 13]{C} and are listed in the following table. We have used the notation of \cite[Chapter 13]{C} for the character labels.

\begin{longtable}{c|c|c|c}
$\textbf{G}$ & 
$e$ &	
$\left(\textbf{L}, \lambda\right)$ &
Is of $\ell$-central defect for 	
\\ \hline \hline


$G_2$ & 	
$1, 2$ &
$\left(G_2, G_2[\theta^i]\right)$ & 
$\ell \neq 3$
 \\ \hline


$F_4$ & 	
$1, 2$ &
$\left(F_4, F_4[\theta^i]\right)$ &
 $\ell \neq 3$\\

$F_4$ & 	
$1, 2$ &
$\left(F_4, F_4[\pm i]\right)^*$  &
$\ell \neq 2$\\ \hline


$E_6$ & 	
$1, 2$ &
$\left(E_6, E_6[\theta^i]\right)$&
$\ell \neq 3$  \\ \hline


${^2E}_6$ & 	
$1, 2$ &
$\left({^2E}_6, {^2E}_6[\theta^i] \right)$ &
$\ell \neq 3$ \\ \hline


$E_7$ & 	
$1$ &
$\left(E_7, E_7[\pm \xi] \right)^{\dagger}$  &
$\ell \neq 2$\\

$E_7$ & 	
$2$ &
$\left(E_7, \phi_{512,11}\right), \left(E_7, \phi_{512,12}\right)$ &
$\ell \neq 2$ \\

$E_7$ & 	
$1$ &
$\left(E_6, E_6[\theta^i]\right)$&
$\ell \neq 3$  \\

$E_7$ & 	
$2$ &
$\left(^2E_6, ^2E_6[\theta^i]\right)$&
$\ell \neq 3$   \\ \hline


$E_8$ & 	
$1, 4$ &
$\left(E_8, E_8[\pm \theta^i] \right)$&
$\ell \neq 2, 3$  \\

$E_8$ & 	
$1,2$ &
$\left(E_8, E_8[\pm i] \right)$ &
$\ell \neq 2$ \\

$E_8$ & 	
$1,2,4$ &
$\left(E_8, E_8[\zeta^j] \right)$ &
$\ell \neq 5$ \\

$E_8$ & 	
$2, 4$ &
$\left(E_8, E_6[\theta^i], \phi_{2,1} \right), \left(E_8, E_6[\theta^i], \phi_{2,2} \right)$ &
$\ell \neq 5$ \\

$E_8$ & 	
$4$ &
\makecell*{
$\left(E_8, E_6[\theta^i], \phi_{1,0} \right), \left(E_8, E_6[\theta^i], \phi_{1,6} \right),$\\
$\left(E_8, E_6[\theta^i], \phi_{1,3'} \right), \left(E_8, E_6[\theta^i], \phi_{1,3''} \right),$ \\
$ \left(E_8, \phi_{4096, 11} \right), \left(E_8, \phi_{4096, 26} \right),$ \\
$  \left(E_8, \phi_{4096, 12} \right), \left(E_8, \phi_{4096, 27} \right),$\\
$\left(E_8, E_7[\pm \xi, 1] \right), \left(E_8, E_7[\pm \xi, \varepsilon] \right)$ }&
every $\ell$ \\

$E_8$ & 	
$1$ &
$\left(E_7, E_7[\pm \xi]\right)$&
$\ell \neq 2$  \\

$E_8$ & 	
$2$ &
$\left(E_7, \phi_{512,11}\right), \left(E_7, \phi_{512,12}\right)^{\ddagger}$ &
$\ell \neq 2$ \\

$E_8$ & 	
$1$ &
$\left(E_6, E_6[\theta^i]\right)$&
$\ell \neq 3$  \\

$E_8$ & 	
$2$ &
$\left({^2E}_6, {^2E}_6[\theta^i] \right)$  &
$\ell \neq 3$\\
\hline

\multicolumn{4}{c}{\makecell*{$\theta := $ exp$(2\pi i / 3)$, $\zeta := $ exp$(2\pi i / 5)$ $\xi := \sqrt{-q}$}} \\
\multicolumn{4}{c}{\makecell*{\small{*\cite{E} omits this pair for $\ell = 3$, $e=2$}  \hspace{2ex} \small{$\dagger$\cite{E} writes $E_7[\pm \zeta]$ instead of $E_7[\pm \xi]$ for $\ell = 2, e = 1$} \\
\small{$\ddagger$\cite{E} writes $E_7[\pm \xi]$ instead of $\phi_{512,11}, \phi_{512,12}$ for $\ell = 5, e = 2$}
}}
\end{longtable}

First suppose that $\ell$ is good for $\textbf{G}$. Then by inspection, the Sylow $\ell$-subgroups of $W_{\textbf{G}^F}(\textbf{L}, \lambda)$ are trivial so by Theorem~\ref{thrm:key} (d), the defect groups of $b$ are isomorphic to a Sylow $\ell$-subgroup of $Z(\textbf{L})^F$. If $\textbf{L} = \textbf{G}$, then the Sylow $\ell$-subgroups of $Z(\textbf{L}^F)$ are trivial by inspection of \cite[Table 24.2]{M/T}. By \cite[Proposition 3.6.8]{C}, since $\textbf{L}$ is connected reductive, $Z(\textbf{L})^F = Z(\textbf{L}^F)$, therefore $b$ has trivial defect and $mf(b) = 1$ by Proposition~\ref{lem:mainlemma} (d). If $\textbf{L}$ and $\textbf{G}$ are such that rk$(\textbf{G}) = $ rk$([\textbf{L},\textbf{L}]) + 1$, then dim($Z^{\circ}\left(\textbf{L})^F\right) = 1$. The Sylow $\ell$-subgroups of $Z^{\circ}(\textbf{L})^F$ are therefore isomorphic to subgroups of the multiplicative group $\textbf{G}_m$, so they are cyclic. By \cite[Proposition 2.2 (i)]{C/E}, since $\ell$ is good for $\textbf{G}$, $Z(\textbf{L})^F_{\ell}$ = $Z^{\circ}(\textbf{L})^F_{\ell}$, therefore $b$ has cyclic defect so $mf(b) = 1$ by Proposition~\ref{lem:mainlemma} (d). 

Now suppose that $\ell$ is bad for $\textbf{G}$, that $\textbf{L} = \textbf{G}$, and that $e=1$. By inspection of the character degrees given in \cite[Chapter 13]{C}, we see that cuspidal characters $\lambda$ of $\textbf{G}^F$ satisfy $\lambda(1)_{\ell}= |\textbf{G}^F|_{\ell}$, so $mf(b) = 1$ by Proposition~\ref{lem:mainlemma} (c). 

The remaining $\ell$-blocks will be handled on a case-by-case basis. First, suppose that $\textbf{G} = E_8$, $\textbf{L} = \phi_1.E_6$ and $\lambda = E_6[\theta^i]$ $(i= 1,2)$ with $\ell \geq 5$ and $e=1$. Then by Theorem~\ref{thrm:Puig}, $k\textbf{G}^Fb$ is Morita equivalent to $kNf$ where $N = N_{\textbf{G}^F} \left( \textbf{L}, \lambda \right)$ and $f = b_{\textbf{L}^F}(\lambda)$ is the block of $k\textbf{L}^F$ containing $\lambda$. Suppose that $P$ is a defect group of $k\textbf{L}^Ff$. Then since $\ell$ is odd and $W_{\textbf{G}^F}(\textbf{L}, \lambda) \cong D_{12}$ is an $\ell'$-group, $P$ is isomorphic to a Sylow $\ell$-subgroup of $Z(\textbf{L})^F$ by Theorem~\ref{thrm:key} (d). Since $N$ normalizes $\textbf{L}$, $P \unlhd N$ so $kNf$ has normal defect. Then by \cite[Theorem 45.12]{Thevenaz}, $kNf $ is Morita equivalent to a twisted algebra $k_{\alpha}(P \rtimes D_{12})$, where $\alpha \in H^2(D_{12}, k^{\times})$. Since $H^2(D_{12}, k^{\times}) \cong C_2$, it follows from the proof of Lemma~\ref{lem:twistedalg} that $mf\left( k_{\alpha} (P \rtimes D_{12})\right)=1$. Whence, $mf(b)=1$. 

Suppose now that $\textbf{G} = E_8$, $\textbf{L} = \phi_1.E_7$, $\lambda = \phi_{512,11}$ or $\phi_{512, 12}$, $\ell=5$ and $e=1$. The relative Weyl group $W_{\textbf{G}^F}\left(\textbf{L}, \lambda \right) \cong S_2$ has no non-trivial Sylow $\ell$-subgroups, so by Theorem~\ref{thrm:MainTheorem} (d) the defect groups of $b$ are isomorphic to a Sylow $\ell$-subgroup of $Z(\textbf{L})^F$. Note that rk$(\textbf{G}) = $ rk$([\textbf{L},\textbf{L}]) + 1$, so dim($Z^{\circ}(\textbf{L})^F) = 1$ and the Sylow $\ell$-subgroups of $Z^{\circ}(\textbf{L})^F$ are cyclic, as above. Again, using \cite[Proposition 2.2]{C/E}, $Z(\textbf{L})^F_{\ell} = Z^{\circ}(\textbf{L})^F_{\ell}$, so $b$ has cyclic defect and $mf(b) = 1$ by Proposition~\ref{lem:mainlemma} (d). 

Suppose that $\textbf{G} = E_7$, $\textbf{L} = \phi_1.E_6(q)$, $\lambda = E_6[\theta^i]$, $ (i=1,2)$, with $\ell=2$ and $e=1$. Then $b$ has dihedral defect by \cite[page 357]{E}. Therefore by Proposition~\ref{lem:mainlemma} (d), $mf(b) = 1$. 

Finally, suppose that we are in one of the following cases: $\textbf{G} = E_8$, $\textbf{L} = \phi_1^2.E_6$, $\lambda = E_6[\theta^i]$, $ (i=1,2)$, with $\ell=2$ and $e=1$; or $\textbf{G} = E_8$, $\textbf{L} = {\phi_2^2}.{^2E}_6$, $\lambda = {^2E_6}[\theta^i]$, $ (i=1,2)$, with $\ell \neq 3$ and $e=2$. From \cite{G} we know that the character field of $\lambda$ is $\mathbb{Q}(\theta)$ where $\theta = $ exp$(\frac{2 \pi i}{3})$. Since $\ell \neq 3$, $\theta$ is an $\ell'$-root of unity so $\hat\sigma(\theta) =\theta^{\ell}$ (see Section~\ref{subsec:blocktheory}). If $\ell \equiv 1$ mod 3, then $\hat\sigma(\theta) = \theta$ so ${^{\hat\sigma}}{\lambda} = \lambda$. Therefore by the arguments of Lemma~\ref{lem:lambdarat}, $mf(b) = 1$. If $ \ell \equiv 2$ mod 3, however, then $\hat\sigma(\theta) = \theta^2 \neq \theta$ so we cannot conclude that $mf(b)=1$. Because $\hat\sigma^2(\theta) = \theta^4 = \theta$, however, it follows that ${^{\hat\sigma^2}}{\lambda} = \lambda$, so $mf(b)$ is at most 2.
\end{proof}

\begin{corollary}
\label{cor:maincor}
Let $\textbf{G}$, $F$ and $k$ be as in Theorem~\ref{thrm:MainTheorem} and suppose that $\textbf{G}^F$ has non trivial centre. Let $Z$ be a central subgroup of $\textbf{G}^F$ and suppose that $\overline{b}$ is a block of $k(\textbf{G}^F/Z)$ dominated by a unipotent block $b$ of $k\textbf{G}^F$. Then $mf(\overline{b})= 1$. 
\end{corollary}

\begin{proof}
The assumption that $\textbf{G}^F$ has non trivial centre means that we do not consider the case where $\textbf{G}^F= E_8(q)$. Thus, for any unipotent block $b$ of $kG$, it follows from the proof of Theorem~\ref{thrm:MainTheorem} that either $\sigma(b) = b$, or $b$ has either trivial, cyclic or dihedral defect.

First suppose that $\overline{b}$ is dominated by a unipotent block $b$ of $kG$ such that $\sigma(b) = b$. Then by Lemma~\ref{lem:domblocks} (b), $\sigma\left(\overline{b}\right)$ is also dominated by $b$. Since $Z$ is central, it then follows from part (c) of Lemma~\ref{lem:domblocks} that $\sigma\left(\overline{b}\right) = \overline{b}$. Therefore $k\left(\textbf{G}^F/Z\right)\overline{b} \cong k\left(\textbf{G}^F/Z\right)\overline{b}^{(\ell)}$ as $k$-algebras by Lemma~\ref{lem:galconj}, so $frob\left(\overline{b}\right) = 1$, hence $mf\left(\overline{b}\right) =~1$. 

Now suppose that $\overline{b}$ is dominated by a unipotent block $b$ of $kG$ which has either trivial, cyclic or dihedral defect. Then by \cite[Ch.5 Theorem 8.7 (ii)]{N/T}, the defect groups of $\overline{b}$ are also either trivial, cyclic or dihedral. Therefore $mf\left(\overline{b}\right) = 1$ by Proposition~\ref{lem:mainlemma} (d).
\end{proof}

\begin{theorem}
\label{thrm:suzree}
Let $G$ be a Suzuki or Ree group in non-defining characteristic. Let $b$ be a block of $kG$, and if $G$ is the large Ree group, assume that $b$ is unipotent. Then $mf(b) = 1$. 
\end{theorem}

\begin{proof}
First let $G$ be the Suzuki group, $^2B_2(q^2)$ ($q=2^{2m+1}$), and let $b$ be a $\ell$-block of $G$ with $\ell \neq 2$. The subgroups of $G$ of odd order are cyclic \cite[Theorem 9]{S}, so $b$ has cyclic defect and therefore $mf(b) = 1$ by Proposition~\ref{lem:mainlemma} (d).

Next let $G$ be the small Ree group, $^2G_2(q^2)$ ($q=3^{2m+1}$), and let $b$ be a 2-block of $G$. The Sylow 2-subgroups of $G$ are elementary abelian of order 8 and \cite[I. 8]{W} shows that the only 2-block of $G$ of full defect is the principal block, which contains the rational valued trivial character. If $b$ is not the principal block, then the defect groups of $b$ are proper subgroups of an elementary abelian group of order 8, so $b$ either has dihedral or cyclic defect. Therefore $mf(b) = 1$ by Proposition~\ref{lem:mainlemma} (b) and (d). 

Now let $G$ be the small Ree group and $\ell \geq 5$, and let $b$ be an $\ell$-block of $G$. The order of $G$ is $|G| = q^6\phi_1\phi_2\phi_4\phi_{12}$ with $q = 3^{2m+1}$ for some $m$. Since $\ell$ divides only one $\phi_i$ for some $i \in \{1, 2, 4, 12\}$, by \cite[Corollary 3.13 (2)]{A/B} the Sylow $\ell$-subgroups of $G$ are cyclic. Therefore $b$ has cyclic defect and $mf(b) = 1$ by Proposition~\ref{lem:mainlemma} (d).

Finally, let $G$ be the large Ree group, $^2F_4(q^2)$ ($q=2^{2m+1}$), and let $b$ be a unipotent $\ell$-block of $G$ with $\ell \neq 2$. By \cite{M}, there are two cases to consider. In the first case we suppose that $\ell \ndivides (q^2 - 1)$. Then $b$ is either the principal block of $G$, or $b$ has trivial defect and therefore $mf(b) = 1$ by Proposition~\ref{lem:mainlemma} (b) and (d). In the second case, suppose that $\ell \divides (q^2 - 1)$. Then $b$ contains one of the following sets of characters (notation as per \cite[Appendix D]{H1}): $\{ \chi_1, \chi_2, \chi_3, \chi_4, \chi_9, \chi_{10}, \chi_{11}\}$, $\{ \chi_5, \chi_7\}$ or $\{ \chi_6, \chi_8\}$. In the first case $b$ is the principal block, and in the second and third cases $b$ has cyclic defect \cite[Appendix D]{H1}, so $mf(b) = 1$ by Proposition~\ref{lem:mainlemma} (b) and (d).
\end{proof}

\section{EXCEPTIONAL COVERING GROUPS}
\label{sec:proof}

\begin{lemma}
\label{lem:lgroupext}
Let $G$ be a finite group. Suppose that there exists a finite group $\hat{G}$ such that $G \unlhd \hat{G}$ and such that for all blocks $B$ of $k\hat{G}$, either $B$ has cyclic defect or $\sigma(B) = B$. Then $mf(b) = 1$ for all blocks $b$ of $kG$. 
\end{lemma}
\begin{proof}
First suppose that $b$ is covered by some block $B$ of $k\hat{G}$ which has cyclic defect. Then the defect groups of $b$ are also cyclic, therefore $mf(b) = 1$ by Proposition~\ref{lem:mainlemma} (d). 

Now suppose that $b$ is covered by a block $B$ of $k\hat{G}$ such that $\sigma(B) = B$. Recall that $B$ covers $b$ if and only if $Bb \neq 0$, and note that this holds if and only if $\sigma(B)\sigma(b) \neq 0$. Therefore, since $\sigma(B) = B$, $\sigma(b)$ is also covered by $B$. Hence $b$ and $\sigma(b)$ are in the same $\hat{G}$-orbit, so there is a group automorphism of $G$ whose induced $k$-algebra automorphism of $kG$ sends $b$ to $\sigma(b)$. Therefore $mf(b) = 1$ by Lemma~\ref{lem:groupaut}. 
\end{proof}

\begin{lemma}
\label{lem:except}
Let $\textbf{G}$ be a simple, simply-connected algebraic group defined over an algebraic closure of the field of $p$ elements. Let $q$ be a power of $p$ and let $F: \textbf{G} \rightarrow \textbf{G}$ be a Steinberg morphism with respect to an $\mathbb{F}_q$-structure with finite group of fixed points, $\textbf{G}^F$. Let $G$ be an exceptional cover of $\textbf{G}^F$ and let $b$ be a block of $kG$. Then $mf (b) = 1$. 
\end{lemma}

\begin{proof}
If $G$ is not from the following list: $3^2.PSU_4(3)$, $3.O_7(3)$, $3.A_6$, $6.A_6$, $3.G_2(3)$, $3.A7$ with $\ell = 2$; or $G = 4^2.PSL_3(4)$ with $\ell=3$; then using GAP \cite{GAP4} and Proposition~\ref{lem:mainlemma}, it can be shown that every block $b$ of $kG$ satisfies at least one of the following three properties: is a principal block, has cyclic defect, or contains a rational valued character, and therefore $mf(b) = 1$. 

If $\ell = 2$ and $G$ is one of $3^2.PSU_4(3)$, $3.O_7(3)$, $3.A_6$, $6.A_6$, $3.G_2(3)$ or $3.A7$, then there are some blocks of $kG$ for which none of these three properties hold. For these groups, however, there exist finite groups $\hat{G}$ such that $G \unlhd \hat{G}$ and such that for every block $B$ of $k\hat{G}$, either $B$ has cyclic defect or $\sigma(B) = B$. Therefore by Lemma~\ref{lem:lgroupext}, $mf(b) = 1$ for all blocks $b$ of $kG$. 

Finally, suppose $G = 4^2.PSL_3(4)$ and $\ell = 3$. Then there are blocks of $kG$ which don't satisfy any of the three properties above, and there is also no suitable $\hat{G}$ which would allow us to apply Lemma~\ref{lem:lgroupext}. Let $G'= PSL_3(4)$ and $Z = C_4 \times C_4$ so $G = Z.G'$. First suppose that $b$ dominates a block $b'$ of $kG'$. Using GAP \cite{GAP4}, as before we can verify that all blocks of $kG'$ satisfy at least one of the three propertie above, so $mf(b')=1$. Since $Z$ is an $\ell'$-group, $b'$ is the unique block dominated by $b$, so by Lemma~\ref{lem:domblocks} (d), $kGb \cong kG'b'$ as $k$-algebras. Therefore $mf(b) = 1$. 

Now suppose that $b$ does not dominate any block of $kG'$. Then by Lemma~\ref{lem:domblocks} (a), $b$ covers a non-principal block of $kZ$. Since $Z$ is an abelian $\ell'$-group, $kZ$ has one linear character in each block. Suppose $b$ covers a block of $kZ$ containing non-trivial character $\mu$, and let $Z_{\mu} = $ ker$\mu$. Then $b$ dominates a unique block $\overline{b}$ of $k(G/Z_{\mu})$, by Lemma~\ref{lem:domblocks} (d), and $mf(b) = mf(\overline{b})$. If $G/Z_{\mu} \cong 2.PSL_3(4)$ then we can once again use GAP \cite{GAP4} to show all blocks of $k(G/Z_{\mu})$ satisfy one of the three properties above, so $mf(\overline{b}) = 1$. If $G/Z_{\mu} \cong 4_1.PSL_3(4)$ or $4_2.PSL_3(4)$, then there are blocks of $k(G/Z_{\mu})$ which don't satisfy any of the three properties. However, in these two cases there exist outer automorphisms of $G/Z_{\mu}$ of order two such that for every block $B$ of $k\left((G/Z_{\mu}).2\right)$, either $B$ has cyclic defect, or $\sigma(B) = B$. Therefore $mf(\overline{b}) = 1$ by Proposition~\ref{lem:mainlemma} and Lemma~\ref{lem:lgroupext}, as required.
\end{proof}

\section{PROOF OF THEOREM~\ref{thrm:maintheorem}}
\label{sec:mainproof}

\begin{proof}
Part (a) is shown in Theorem~\ref{thrm:alt}. The result follows for exceptional covering groups of finite groups of Lie type by Lemma~\ref{lem:except}. The remainder of part (b) follows from Corollary~\ref{corr:defining}. Part (c) is shown for the Suzuki and Ree groups in Theorem~\ref{thrm:suzree}, and for all remaining cases in Corollary~\ref{cor:maincor}. 
\end{proof}

\bibliographystyle{acm}
\bibliography{BIBLIOGRAPHY}

\begin{thebibliography}{10}

\bibitem{A/B}
{\sc Alperin, J., and Brou{\'e}, M.}
\newblock Local methods in block theory.
\newblock {\em Ann. of Math. (2) 110}, 1 (1979), 143--157.

\bibitem{B/K}
{\sc Benson, D., and Kessar, R.}
\newblock Blocks inequivalent to their {F}robenius twists.
\newblock {\em J. Algebra 315}, 2 (2007), 588--599.

\bibitem{B/Mi}
{\sc Bonnaf{\'e}, C., and Michel, J.}
\newblock Computational proof of the {M}ackey formula for {$q>2$}.
\newblock {\em J. Algebra 327\/} (2011), 506--526.

\bibitem{B/M}
{\sc Brou{\'e}, M., and Malle, G.}
\newblock Generalized {H}arish-{C}handra theory.
\newblock In {\em Representations of reductive groups}, Publ. Newton Inst.
  Cambridge Univ. Press, Cambridge, 1998, pp.~85--103.

\bibitem{B/M/M}
{\sc Brou{\'e}, M., Malle, G., and Michel, J.}
\newblock Generic blocks of finite reductive groups.
\newblock {\em Ast\'erisque}, 212 (1993), 7--92.
\newblock Repr{\'e}sentations unipotentes g{\'e}n{\'e}riques et blocs des
  groupes r{\'e}ductifs finis.

\bibitem{Ca}
{\sc Cabanes, M.}
\newblock Local structure of the {$p$}-blocks of {$\widetilde{S}_n$}.
\newblock {\em Math. Z. 198}, 4 (1988), 519--543.

\bibitem{C/E2}
{\sc Cabanes, M., and Enguehard, M.}
\newblock Unipotent blocks of finite reductive groups of a given type.
\newblock {\em Math. Z. 213}, 3 (1993), 479--490.

\bibitem{C/E}
{\sc Cabanes, M., and Enguehard, M.}
\newblock On unipotent blocks and their ordinary characters.
\newblock {\em Invent. Math. 117}, 1 (1994), 149--164.

\bibitem{C}
{\sc Carter, R.~W.}
\newblock {\em Finite groups of {L}ie type}.
\newblock Wiley Classics Library. John Wiley \& Sons, Ltd., Chichester, 1993.
\newblock Conjugacy classes and complex characters, Reprint of the 1985
  original, A Wiley-Interscience Publication.

\bibitem{C/R}
{\sc Curtis, C.~W., and Reiner, I.}
\newblock {\em Methods of representation theory. {V}ol. {I}}.
\newblock Wiley Classics Library. John Wiley \& Sons, Inc., New York, 1990.
\newblock With applications to finite groups and orders, Reprint of the 1981
  original, A Wiley-Interscience Publication.

\bibitem{D/M}
{\sc Digne, F., and Michel, J.}
\newblock {\em Representations of finite groups of {L}ie type}, vol.~21 of {\em
  London Mathematical Society Student Texts}.
\newblock Cambridge University Press, Cambridge, 1991.

\bibitem{E}
{\sc Enguehard, M.}
\newblock Sur les {$l$}-blocs unipotents des groupes r\'eductifs finis quand
  {$l$} est mauvais.
\newblock {\em J. Algebra 230}, 2 (2000), 334--377.

\bibitem{Erd}
{\sc Erdmann, K.}
\newblock {\em Blocks of tame representation type and related algebras},
  vol.~1428 of {\em Lecture Notes in Mathematics}.
\newblock Springer-Verlag, Berlin, 1990.

\bibitem{GAP4}
{\sc The GAP~Group}.
\newblock {\em {GAP -- Groups, Algorithms, and Programming, Version 4.7.9}},
  2015.

\bibitem{G}
{\sc Geck, M.}
\newblock Character values, {S}chur indices and character sheaves.
\newblock {\em Represent. Theory 7\/} (2003), 19--55 (electronic).

\bibitem{H1}
{\sc Hiss, G.}
\newblock Zerlegungszahlen endlicher gruppen vom lie-typ in nicht-definierender
  charakteristik.
\newblock Habilitationsschrift, 1990.

\bibitem{H}
{\sc Humphreys, J.~E.}
\newblock {\em Modular representations of finite groups of {L}ie type},
  vol.~326 of {\em London Mathematical Society Lecture Note Series}.
\newblock Cambridge University Press, Cambridge, 2006.

\bibitem{J/K}
{\sc James, G., and Kerber, A.}
\newblock {\em The representation theory of the symmetric group}, vol.~16 of
  {\em Encyclopedia of Mathematics and its Applications}.
\newblock Addison-Wesley Publishing Co., Reading, Mass., 1981.
\newblock With a foreword by P. M. Cohn, With an introduction by Gilbert de B.
  Robinson.

\bibitem{K2}
{\sc Kessar, R.}
\newblock Blocks and source algebras for the double covers of the symmetric and
  alternating groups.
\newblock {\em J. Algebra 186}, 3 (1996), 872--933.

\bibitem{K}
{\sc Kessar, R.}
\newblock A remark on {D}onovan's conjecture.
\newblock {\em Arch. Math. (Basel) 82}, 5 (2004), 391--394.

\bibitem{K3}
{\sc Kessar, R.}
\newblock On isotypies between {G}alois conjugate blocks.
\newblock In {\em Buildings, finite geometries and groups}, vol.~10 of {\em
  Springer Proc. Math.} Springer, New York, 2012, pp.~153--162.

\bibitem{K/M}
{\sc Kessar, R., and Malle, G.}
\newblock Quasi-isolated blocks and {B}rauer's height zero conjecture.
\newblock {\em Ann. of Math. (2) 178}, 1 (2013), 321--384.

\bibitem{K/M2}
{\sc Kessar, R., and Malle, G.}
\newblock Lusztig induction and $l$-blocks of finite reductive groups.
\newblock 2015.

\bibitem{Block}
{\sc Linckelmann, M.}
\newblock The block theory of finite group algebras.

\bibitem{M}
{\sc Malle, G.}
\newblock Die unipotenten {C}haraktere von {${}^2F_4(q^2)$}.
\newblock {\em Comm. Algebra 18}, 7 (1990), 2361--2381.

\bibitem{M/T}
{\sc Malle, G., and Testerman, D.}
\newblock {\em Linear algebraic groups and finite groups of {L}ie type},
  vol.~133 of {\em Cambridge Studies in Advanced Mathematics}.
\newblock Cambridge University Press, Cambridge, 2011.

\bibitem{Mo}
{\sc Morris, A.~O.}
\newblock The spin representation of the symmetric group.
\newblock {\em Proc. London Math. Soc. (3) 12\/} (1962), 55--76.

\bibitem{N/T}
{\sc Nagao, H., and Tsushima, Y.}
\newblock {\em Representations of finite groups}.
\newblock Academic Press, Inc., Boston, MA, 1989.
\newblock Translated from the Japanese.

\bibitem{O2}
{\sc Olsson, J.~B.}
\newblock {\em Combinatorics and representations of finite groups}, vol.~20 of
  {\em Vorlesungen aus dem Fachbereich Mathematik der Universit\"at GH Essen
  [Lecture Notes in Mathematics at the University of Essen]}.
\newblock Universit\"at Essen, Fachbereich Mathematik, Essen, 1993.

\bibitem{P}
{\sc Puig, L.}
\newblock On {J}oanna {S}copes' criterion of equivalence for blocks of
  symmetric groups.
\newblock {\em Algebra Colloq. 1}, 1 (1994), 25--55.

\bibitem{S}
{\sc Suzuki, M.}
\newblock On a class of doubly transitive groups.
\newblock {\em Ann. of Math. (2) 75\/} (1962), 105--145.

\bibitem{Thevenaz}
{\sc Th{\'e}venaz, J.}
\newblock {\em {$G$}-algebras and modular representation theory}.
\newblock Oxford Mathematical Monographs. The Clarendon Press, Oxford
  University Press, New York, 1995.
\newblock Oxford Science Publications.

\bibitem{W}
{\sc Ward, H.~N.}
\newblock On {R}ee's series of simple groups.
\newblock {\em Trans. Amer. Math. Soc. 121\/} (1966), 62--89.

\end{thebibliography}

\end{document}